\documentclass[12pt]{article}
\usepackage{amsmath,amsthm}
\usepackage{amsfonts}
\usepackage{amssymb}
\usepackage{enumerate}
\usepackage{hyperref}
\usepackage{comment}
\newtheorem{theorem}{Theorem}[section]
\newtheorem{definition}[theorem]{Definition}
\newtheorem{lemma}[theorem]{Lemma}

\newtheorem{cor}[theorem]{Corollary}
\newtheorem{prop}[theorem]{Proposition}

\newtheorem{example}[theorem]{Example}

\def\eq{\begin{equation}}
\def\equ{\end{equation}}

\def\h1h{H^1(H)}
\def\h1l2{H^1(\ell^2)}

\def\h1x{H^1(X, d, \mu )}

\font\textmsbm= msbm10 scaled 1200
\font\scriptmsbm= msbm7 scaled 1200
\font\scriptscriptmsbm= msbm5 scaled 1200

\def\Id{\mathchoice{\mbox{Id}}{\mbox{Id}}{\mbox{\scriptsize
Id}}{\mbox{\tiny Id}}}

\def\endproof{\par\nobreak\hbox to \hsize{\hfil\vrule width 5pt height
5pt}\goodbreak\vskip 3pt}
\def\endproofof{\par\nobreak\hbox to \hsize{\hfil\vrule width 5pt height
5pt}\goodbreak\vskip 3pt}

\newfam\tmwfama
\textfont\tmwfama\textmsbm
\scriptfont\tmwfama\scriptmsbm
\scriptscriptfont\tmwfama\scriptscriptmsbm
\def\bbb{\mathbb}

\def\bR{{\bbb R}}
\def\bC{{\bbb C}}

\def\bE{{\bbb E}}

\def\bK{{\bbb K}}

\def\s{\sigma}

      \def\sbe{\subseteq}

\def\cA{{\cal A}}

\def\cC{{\cal C}}
\def\cD{{\cal D}}
\def\cE{{\cal E}}

\def\cK{{\cal K}}
\def\cL{{\cal L}}

\def\Id{\mbox{ \rm Id} }

\newcommand{\sptext}[3]{\hspace{#1 em}\mbox{#2}\hspace{#3 em}}

\newcommand{\umd}{{\tt UMD}}
\newcommand{\ftn}{\mathcal{F}}
\newcommand{\ftnd}{\mathcal{F}^{}}
\newcommand{\noo}{\left \|}
\newcommand{\rrm}{\right \|}
\newcommand{\equa}{\begin{eqnarray*}}

\newcommand{\tion}{\end{eqnarray*}}

\newcommand{\prob}{\mathbb{P}}
\newcommand{\ev}{\mathbb{E}}

\def\y{\lbrack\!\lbrack} 
\def\yy{\rbrack\!\rbrack}
\def\s{\sigma}

\def\sbe{\subseteq}
\def\bbb{\mathbb}
\def\bE{{\bbb E}}

\allowdisplaybreaks

\begin{document}

\title{Extrapolation of vector valued rearrangement operators}
\author{Stefan Geiss 
        \thanks{Supported by the project
         {\em Stochastic and Harmonic Analysis, Interactions and Applications}
         of the Academy of Finland.}
        \and
        Paul F.X. M\"uller} 
\maketitle
\begin{abstract} 
Given an injective map $\tau:\cD\to\cD$ between the dyadic intervals of the
unit interval $[0,1)$, we study extrapolation properties of the induced
rearrangement operator of the Haar system 
$\Id _X \otimes T_{p,\tau}: L_{X,0}^p([0,1)) \to L_X^p([0,1))$, where 
$X$ is a Banach space and $L_{X,0}^p$ the subspace of mean zero 
random variables. If $X$ is a $\umd$-space, then we prove that the property 
that $\Id _X \otimes T_{p,\tau}$ is an isomorphism for some 
$1 <p  \ne 2 <\infty$ 
extrapolates across the entire scale of  $L_{X}^q$-spaces with $1 <q<\infty.$
In contrast, if  only $\Id _X \otimes T_{p,\tau}$ is bounded and not its 
inverse, then we show that there can only exist one-sided extrapolation 
theorems. 
\end{abstract}





\section{Introduction} 

In vector valued $L^p$-spaces we study rearrangement operators of the 
system
\[ \{ h_I/ |I|^{1/p} : I \in \cD \} , \]
where $\cD $ denotes the collection of all dyadic intervals included in 
$ [0,1)$ and $h_I$ is the $L_\infty$-normalized Haar function with
support $I$. 
These rearrangement operators are defined  by an injective map 
$\tau : \cD \to \cD $ as extension of
\[ \Id _{X} \otimes T_{p,\tau} : \sum_{ I \in \cD}  a_I  h_I/ |I|^{1/p} \to 
   \sum_{ I \in \cD } a_I  h_{\tau(I)}/ |\tau(I)|^{1/p}, \]
where $(a_I)_{I\in\cD}\subseteq X$ is finitely supported and $X$ is a Banach space.
This paper continues  \cite{pfxm1} and is related in spirit to \cite{gmp}. 
In particular, we are motivated by extrapolation properties of vector valued 
martingale transforms, i.e. maps of type
\begin{equation}\label{eqn:MT}
 \sum_{ I \in \cD}  a_I  h_I \to \sum_{ I \in \cD } c_I a_I h_I
\end{equation}
where  $(a_I)_{I\in \cD}\subseteq X$ is finitely supported
and $(c_I)_{I\in\cD}  \in \ell_\infty(\cD). $
Extrapolation theorems for these martingale transforms were widely 
studied in the literature and go back, for example, to Maurey \cite{Mau} and 
Burkholder-Gundy \cite{Bur-G} (see \cite{Burk3} for a general overview). 
In our setting these classical
theorems state that if (\ref{eqn:MT}) is bounded on  $L^p_X $  for some $ p\in (1,\infty) $, 
then it is bounded 
on  $L^q_X $ for all  $q\in (1,\infty)$. The significance of those theorems can be 
already seen in the scalar valued setting: Since a martingale transform 
is trivially bounded on $L^2 ,$ extrapolation yields its boundedness 
on each of the spaces  $L^q$ with $q\in (1,\infty)$.
The aim of this paper is to analyze the extrapolation properties of the family 
$\Id _{X} \otimes T_{p,\tau}.$


In Section \ref{sec:examples} we start by two examples.
Example \ref{example:rearrangement_umd} shows 
that the continuity of a 'typical' permutation $\Id _{X} \otimes T_{p,\tau}$ already 
implies that $X$ has to have the \umd-property. The second example provides a permutation 
such that the continuity of $\Id _{X} \otimes T_{p,\tau}$ with $p\in (1,2]$ implies the 
type $p$ property  of the Banach space $X$. As a consequence we deduce in 
Corollary \ref{corollary:example} that one does not have an upwards extrapolation: 
For $X=\ell_p$ and $p\in (1,2)$ (so that $X$ is, in particular, a \umd-space)
there is a permutation $\tau$ such that $\Id _{X} \otimes T_{p,\tau}$ is continues, but
$\Id _{X} \otimes T_{q,\tau}$ fails to be continuous for $q\in (p,2]$.

The natural question arises whether we still have a one-sided extrapolation meaning that 
the boundedness of $\Id _{X} \otimes T_{p,\tau}$ implies that one of
$\Id _{X} \otimes T_{q,\tau}$ in the case $1<q<p<2$.  


In Section \ref{sec:maurey} we answer this to the  positive for permutations $\tau$
satisfying the assumption $|\tau(I)|=|I|$. The results are formulated in 
Theorem \ref{theorem:one_sided_extrapolation_new_new} and 
Corollary \ref{cor:one_sided_extrapolation_under_UMD} and proved by 
transferring Maurey's classical argument \cite{Mau} to the permutation case via 
Proposition \ref{proposition:maximal_inequality_condition_Semenov}.
In Corollary \ref{cor:one_sided_extrapolation_under_UMD} we extrapolate 
the boundedness of $\Id _{X} \otimes T_{p,\tau}$ for a $\umd$-space $X$ and $p \in (1,2)$
{\em downwards to 1} to the boundedness of 
$\Id _{X} \otimes T_{q,\tau}$ for $q\in (1,p)$. 


In Section \ref{sec:carleson} we do not assume anymore the condition 
$|\tau(I)|=|I|$. In Corollaries \ref{corollary:downwards_extrapolation} and
\ref{corollary:downwards_extrapolation_semenov} we obtain a one-sided extrapolation as well.
By duality Corollary  \ref{corollary:downwards_extrapolation} yields a two-sided 
extrapolation in Theorem \ref{theorem:extraisolp}:
We show for a \umd-space $X$ that if $\Id _{X} \otimes T_{p,\tau}  $ is an isomorphism 
on some  $L^p_{X,0} $ 
with $ 1 < p \ne 2 < \infty$, then
the rearrangement $\Id _{X} \otimes T_{q,\tau}$ is an isomorphism on $L^q_{X,0}$
for each  $q\in (1,\infty)$.
Thus for a $\umd$-space valued  rearrangement  the property of 
being an isomorphism extrapolates across the entire scale of $L^q_{X,0} $ spaces, 
$q\in (1,\infty)$ -- just as for martingale transforms or for scalar valued 
rearrangements $T_{q,\tau}: L^q_0 \to L^q_0, $   see \cite{pfxm1}.

The extrapolation properties of {\it scalar} valued rearrangement operators
are a direct consequence of Pisier's re-norming of $H^1, $ 
\[       \|g\|_{H^1}^{1-\theta}  
    \sim \sup \{ \| \sum |g_I|^{1-\theta} |w_I|^\theta h_I \|_{L^p} : \|w\|_{L^2} = 1 \},\]
where $p\in (1,2)$, $1/p = 1-(\theta/2),$ 
$g = \sum g_I h_I ,$  and $w = \sum w_I h_I .$
This well known fact is recorded for instance in \cite{pfxm2}
and was  exploited further in \cite{gmp}.
As Pisier's re-norming of $H^1$ uses the lattice structure of 
${L^p},$ our analysis of the 
{\it vector} valued case circumvents its  use  and 
relies  instead on combinatorial and geometric properties of
$\tau$ that hold when $T_{p,\tau}$ is an isomorphism \cite{pfxm1}.

 
\section{Preliminaries}
In the following we equip the unit interval $[0,1)$ with the Lebesgue measure
$\lambda$.
The set of dyadic intervals of length $2^{-k}$ is denoted by $\cD_k$, the set of all 
dyadic intervals by $\cD$, and 
$\ftnd_k := \sigma (\cD_k)$. 
Given $I\in \cD$, we use $Q(I):= \{ K \subseteq I : K \in \cD \}$
and  $h_I$ denotes the $L_\infty$-normalized Haar
function supported on $I$. 
For a Banach space $X$ we let $L^p_{X}= L^p_{X}([0,1))$ be the
space of all Radon random variables $f:[0,1)\to X$ such that
$\| f \|_{L_X^p}^p := \int_0^1 \| f(t)\|_X^p dt <\infty$
and $L^p_{X,0}$ be the sub-space of mean zero random variables, where 
$L^p= L_{\bK}^p([0,1))$ and $L_0^p= L_{\bK,0}^p([0,1))$ if nothing is said
to the contrary with $\bK\in \{\bR,\bC\}$. 
To avoid artificial special cases we assume that the Banach spaces 
are at least of dimension one.

\paragraph{Spaces of type and cotype.}
Let $1 \le p \le 2 \le q < \infty$.
A Banach space $X$ is of {\em type $p$} ({\em cotype $q$}) provided that
there is a constant $c>0$ such that for all $n=1,2,...$ and 
$a_1,a_2,...,a_n\in X$ one has that
\[      \left \|  \sum_{k=1}^n r_k a_k \right \|_{L^p_X} 
    \le c \left ( \sum_{k=1}^n \| a_k \|^p_X \right )^\frac{1}{p}
    \left ( 
 \left ( \sum_{k=1}^n \| a_k \|^q_X \right )^\frac{1}{q}
   \le c
     \left \|  \sum_{k=1}^n r_k a_k \right \|_{L^q_X} \right )
, \]
where $r_1,r_2,...$ denote independent Bernoulli random variables.
We let ${\rm Type _p(X)} := \inf c$
(${\rm Cotype _q(X)} := \inf c$).

\paragraph{\umd-spaces.}
A Banach space $X$ is called {\em \umd-space} provided that for some $p\in (1,\infty)$
(equivalently, for all $p\in (1,\infty)$) there is a constant $c_p>0$ such that
\[         \sup_{\theta_k\in [-1,1]} 
           \noo \sum_{k=1}^n \theta_k d_k \rrm_{L^p_X}
   \le c_p \noo \sum_{k=1}^n          d_k \rrm_{L^p_X} \]
for all $n=1,2,...$ and all martingale difference sequences 
$(d_k)_{k=1}^n \subseteq L_X^1(\ftnd_n)$ with respect to $(\ftnd_k)_{k=0}^n$, i.e.
$d_k$ is $\ftnd_k$-measurable and
$\bE (d_k|\ftnd_{k-1})=0$ for $k=1,...,n$. The infimum of all possible 
$c_p>0$ is denoted by $\umd_p(X)$.
\medskip

Using \cite[page 12]{Burk4} it follows that $\umd_p(X)=\inf d_p$,
where the infimum is taken over all $d_p>0$ such that
\[         \sup_{\theta_I\in [-1,1]} 
           \noo \sum_{I\in\cD} \theta_I a_I h_I \rrm_{L^p_X}
   \le d_p \noo \sum_{I\in\cD}          a_I h_I \rrm_{L^p_X} \]
for all finitely supported $(a_I)_{I\in\cD}\subseteq X$.
An overview about \umd-spaces can be found in \cite{Burk3}. 

\paragraph{Hardy spaces.}
We recall the definition of Hardy spaces we shall use.

\begin{definition}\label{definition:hardy_spaces} \rm
\begin{enumerate}[(i)]
\item A function $a\in L^1_{X,0}(\ftnd_N)$, where $N\ge 1$, is called 
      {\em atom} provided there exists a stopping time 
      $\nu:\Omega\to \{ + \infty,0,...,N \}$ such that
      \begin{enumerate}[{\rm (a)}]
      \item $a_n:= \bE (a|\ftn_n) = 0$ on $\{ n \le \nu \}$ for $n=0,...,N$,
      \item $\| a \|_{L^\infty_X} \prob(\nu < \infty ) \le 1$.
      \end{enumerate}
\item The space $H^{1,at}_X(\ftnd_N)$  is given by the norm
      \[ \| f \|_{H^{1,at}_X} := \inf \sum_{k=1}^\infty |\mu_k|, 
         \quad 
         f \in L_{X,0}^1(\ftnd_N), \]
      where the infimum is taken over all sequences 
      $(\mu_k)_{k=1}^\infty \subset [0,\infty)$ and atoms
      $(a^k)_{k=1}^\infty$ such that
      $f = \sum_{k=1}^\infty \mu_k a^k$ in $L^1_X(\ftnd_N)$. 
\item Given $p\in [1,\infty)$, the space $H^p_X(\ftnd_N)$ is given by the norm
      \[    \| f \|_{H^p_X} 
         := \left ( \bE \sup_{n=0,...,N} \| \bE (f|\ftnd_n) \|_X^p\right )^{\frac{1}{p}}, 
         \quad 
         f \in L_{X,0}^p(\ftnd_N). \]
\end{enumerate}
\end{definition}
\bigskip

For an atom $a$ we have that $a=0$ on $\{ \nu = \infty \}$,
${\rm supp}(a) \subseteq \{ \nu < \infty \}$, and
\[ \bE \|a\|_X \le \| a \|_{L^\infty_X} \prob(\nu < \infty)
               \le 1. \]
The following inequality is well-known 
(see \cite{Bernard-Mais} and \cite{Coifman}, cf. \cite{Weisz}): 
\begin{equation}
\label{eqn:H1_H1at}
         \| f \|_{H^1_X(\ftnd_N)} 
   \le    \| f \|_{H^{1,at}_X(\ftnd_N)} 
   \le 18 \| f \|_{H^1_X(\ftnd_N)}. 
\end{equation}

\paragraph{Rearrangement operators.}
Let $\tau : \cD \to \cD $ be an injective map.
Given a Banach space $X$ and $p\in [1,\infty)$, we define the rearrangement 
operator $\Id_X \otimes T_{p,\tau}$ on finite linear combinations of Haar 
functions as 
\[ \Id _{X} \otimes T_{p,\tau} : \sum  a_I\frac{h_I}{|I|^{1/p}} \to 
   \sum a_I  \frac{h_{\tau(I)}}{|\tau(I)|^{1/p} }, \quad a_I \in X, \]
and let
\[
      \| \Id_X \otimes T_{p,\tau} \| 
  := \sup \left \{ \left 
          \| \sum_{I\in\cD} a_I \frac{h_{\tau(I)}}{|\tau(I)|^{1/p} }
          \right \|_{L_X^p} : \left \| \sum_{I\in\cD} a_I \frac{h_I}{|I|^{1/p} } 
                              \right \|_{L_X^p} \le 1 \right \} \]
where the supremum is taken over all finitely supported 
$(a_I)_{I\in\cD}\subseteq X$.
In the case $ \| \Id_X \otimes T_{p,\tau} \| < \infty$ we say that
$\Id_X \otimes T_{p,\tau}$ is bounded because it can be 
continuously extended to $L_{X,0}^p([0,1))\to L_X^p([0,1))$. 
The dependence on $p$ of the operator $T_{p,\tau}$ disappears when the 
injection  $\tau : \cD \to \cD $ satisfies 
\[  |\tau(I)|= |I|,\quad I \in \cD,  \]
so that we also use $T_{\tau}=T_{p,\tau}.$ 
 
\paragraph{Semenov's condition.}
For a non-empty collection $\mathcal{C}$ of dyadic intervals we let
$ \mathcal{C}^*:= \bigcup_{I\in \mathcal{C}} I . $
A rearrangement  $\tau:\cD\to\cD$ with 
\[ |\tau(I)|=|I| \]
satisfies {\em Semenov's condition} if there is a $\kappa\in [1,\infty)$ such that
\begin{equation}\label{eqn:Semenov_condition}
    \sup_{\cC \sbe \cD} \frac{|\tau(\cC) ^* |}{|\cC ^* |} 
\le \kappa 
 <  \infty.
\end{equation}
Given $p\in (1,2)$, Semenov's theorem \cite{Sem,Sem-Stoeckert}
asserts that  under the restriction $|\tau(I)|=|I|$,
condition \eqref{eqn:Semenov_condition} 
is  equivalent to the boundedness of
$T_\tau : L^p_0([0,1)) \to L^{p}([0,1))$.

\paragraph{Carleson's constant.}
For a non-empty collection  $\cE \sbe \cD$ the {\em Carleson constant} is 
given by 
\[    \y \cE \yy 
   := \sup_{I\in \cE}\frac{1}{|I|} \sum_{J\sbe I, \, J \in \cE} |J| . \]
The Carleson constant is linked  to rearrangement operators by the following 
theorem \cite[Theorems 2 and 3]{pfxm1}:
For a bijection $\tau:\cD\to\cD$
the assertion that for some (all) $p\in (1,\infty)$ with $p\not = 2$ one has 
\[      \| \Id_\bK \otimes T_{p,\tau}: L^p_{X,0} \to L^p_X \|
         \cdot \|\Id_\bK \otimes T_{p,\tau^{-1}}: L^p_{X,0} \to L^p_X \|
     <   \infty \]
is equivalent to the existence of an
$A\ge 1$ such that
\[     \frac{1}{A}  \y      \cE   \yy
   \le              \y \tau(\cE ) \yy
   \le           A  \y      \cE   \yy \] 
for all non-empty $\cE\subseteq \cD$.


\section{Two  examples}
\label{sec:examples}

In this section we consider bijections $\tau:\cD\to\cD$ such
that $|\tau(I)|=|I|$ for all $I\in\cD$ and provide examples which 
show that $\umd_p (X) $ and ${\rm Type}_p (X)$
may both be obstructions to the boundedness of 
\[ \Id_X \otimes T_{\tau} : L^{p}_{X,0} \to L^{p}_X. \]
>From that it becomes clear 
that Semenov's boundedness criterion \cite{Sem}
does not have a direct correspondence in  
the vector valued case.  

\begin{example} \label{example:rearrangement_umd}\rm
Let $\tau_0 : \cD \to \cD $ be the injection that 
leaves invariant the intervals of the even numbered dyadic levels. 
On the odd numbered dyadic levels we define $\tau_0$
to exchange   the dyadic intervals contained in $[0,1/2) $  with  
those contained in $ [1/2,1)$ by
the  shifts 
\[  \tau_0  (I) = I + \frac12  \text{  if  } I \sbe  [0,1/2)
    \text{  and  } 
    \tau_0  (I) = I -\frac12  \text{  if  } I \sbe  [1/2,1).\]
Then one has the following:
\begin{enumerate}[(i)]
\item The rearrangement $\tau_0=\tau_0^{-1}$ 
      satisfies Semenov's condition with $\kappa=2$ so that
      $T_{\tau_0}$ is an isomorphism on $L^p_0$ for $p\in (1,\infty). $
\item For $p\in (1,\infty)$ one has
      \begin{equation}\label{eqn:umd-permuation}
           \frac{1}{3}\umd_p (X) 
      \le  \|\Id_X \otimes T_{\tau_0} : L^p_{X,0} \to L^p_X\|
      \le  2\umd_p (X)
      \end{equation}
      so that the boundedness of $\Id_X \otimes T_ {\tau_0}$ on $L^p_{X,0}$,
      $p\in (1,\infty)$, holds precisely when $X$ satisfies the 
      \umd-property. 
\end{enumerate}
\end{example} 
\begin{proof}
Assertion (i) is obvious so that let us turn to (ii) and 
let $N\ge 2$ be even and recall that $\cD_k$ is the set of dyadic
intervals of length $2^{-k}.$ For  $k\ge 1$  define 
\[ \cD_k^- := \{ I\in \cD_k : I \subseteq [0,1/2) \}. \]
The testing functions by which we link  the boundedness 
of  $\Id_X \otimes T_{\tau_0}$  to  the $\umd$-property of $X$ are
\[ f= \sum_{k=1}^N   \sum_{I\in\cD_k^-} a_I h_I 
   \quad\text{and}\quad
   g = \sum_{k=1}^{N/2}  \sum_{I\in\cD_{2k}^-} a_I h_I , \]
where $a_I\in X.$ 
Note that $g$ is obtained from $f$ by deleting every 
second dyadic level from the Haar expansion of $f$ starting with level 1.
Consequently,
\equa
      \left \| \sum_{k=1}^N (-1)^k \sum_{I\in\cD_k^-}a_I h_I
      \right \|_{L_X^p}
& = & \left \| f - 2 g \right \|_{L_X^p} \\
&\le& \| f \|_{L_X^p} + 2 \| g \|_{L_X^p} \\
&\le& \| f \|_{L_X^p} + 2 \| (\Id_X \otimes T_{\tau_0}) f \|_{L_X^p} \\
&\le& \left ( 1 + 2 \|\Id_X \otimes T_{\tau_0} : L^p_{X,0} \to L^p_X\| \right )
      \| f \|_{L_X^p}  \\
&\le& 3 \|\Id_X \otimes T_{\tau_0} : L^p_{X,0} \to L^p_X\|
      \left \| \sum_{k=1}^N   \sum_{I\in\cD_k^-}a_I h_I 
      \right \|_{L_X^p}\!\!.
\tion
In our definition of $\umd_p(X)$ it is sufficient to consider $\pm 1$ 
transforms (this is a well-known extreme point
argument). Furthermore, by an appropriate augmentation of the filtration 
we can even restrict ourselves to alternating  sequences of
signs $\pm 1$. Hence we obtain the left hand side of 
\eqref{eqn:umd-permuation} (in fact, we can think to work on $[0,1/2)$ as probability 
space after re-normalization).

For the right hand side of  \eqref{eqn:umd-permuation} we fix some $N\ge 1$ and 
observe that the action of the 
above rearrangement is an isometry  when restricted to 
$\sum_{k \mbox{ odd},0\le k\le N} \sum_{k\in\cD_k}a_I h_I$ and an isometry  when restricted to 
$\sum_{k \mbox{ even},0 \le k\le N} \sum_{k\in\cD_k} a_I h_I$.
Using the $\umd$-property of $X$,  we merge this information
to obtain the boundedness of the 
rearrangement operator on the entire space  $L_{X,0}^p.$ 
\end{proof}
 
\begin{example} \label{example:rearrangement_type}\rm
There exists a rearrangement $\tau_0 : \cD \to \cD $ with
$ |\tau_0(I)| = |I|$ satisfying the Semenov condition 
(\ref{eqn:Semenov_condition}), 
such that for all $p\in (1,2]$ and all Banach spaces $X$ one has that
\[ \text{Type}_p(X) \le \|\Id_X \otimes T_{\tau_0} : L^p_{X,0} \to L^p_X\|. \]
\end{example}

\begin{proof}
(a) Fix $n\ge 1$ and assume disjoint dyadic intervals $I_0,...,I_n$ of
the same length, one after each other starting with $I_0$. 
Let 
\[ \cA_k := \left\{ I\in\cD : I \subseteq I_k, |I|=2^{-k}|I_k|
            \right \} \]
for $k=1,...,n$. We define a permutation $\tau_n:\cD\to\cD$ such that
\begin{enumerate}[(i)]
\item $\cA_k$ is shifted from $I_k$ to $I_0$ for each
      $k=1,...,n$,
\item all subintervals of $I_0$ of length $2^{-k}|I_0|$, $k=1,...,n$, are
      shifted to $I_1$,
\item all subintervals of $I_1$ of length $2^{-k}|I_1|$, $k=2,...,n$, are
      shifted to $I_2$, \\...,
\item all subintervals of $I_{n-1}$ of length $2^{-n}|I_{n-1}|$ are
      shifted to $I_n$.
\end{enumerate}
On all other intervals $\tau_n$ acts as an identity. One can check
that $\tau_n$ satisfies Semenov's condition with $\kappa=3$.
Moreover, for $a_1,...,a_n\in X$,
\equa
      \int_0^1 \noo \sum_{k=1}^n \sum_{I\in \cA_k} a_k h_I(t) \rrm^p_X dt
& = & |I_0| \sum_{k=1}^n \| a_k \|^p_X, \\
      \int_0^1 \noo \sum_{k=1}^n \sum_{I\in \cA_k} a_k h_{\tau_n(I)}(t)  
               \rrm^p_X dt
& = & |I_0| \noo \sum_{k=1}^n r_k a_k \rrm_{L^p_X}^p,
\tion
so that
\[     \noo \sum_{k=1}^n r_k a_k \rrm_{L^p_X}
   \le \| \Id_X \otimes T _{\tau_n} : L^p_{X,0} \to  L^p_X \|
       \left ( \sum_{k=1}^n \| a_k \|^p\right )^\frac{1}{p} \]
where $r_1,...,r_n$ are independent Bernoulli random variables.
\medskip

(b) Now we 'glue together' the permutations $\tau_1,\tau_2,...$:
to this end we find pairwise disjoint dyadic  intervals 
$I_0^1,I_1^1\subseteq [0,1/2)$, \hspace{.1em}
$I_0^2,I_1^2,I_2^2 \subseteq [1/2,3/4)$, \hspace{.1em}
$I_0^3,I_1^3,I_2^3,I_3^3\subseteq [3/4,7/8)$,\ldots, where $I_0^n,...,I_n^n$ 
is a collection as in part (a). Defining the permutation $\tau_0$ on
$I_0^n,...,I_n^n$ as in (a) for all $n=1,2,....$ and elsewhere as 
identity, we arrive at our desired permutation $\tau_0$.
\end{proof}
\medskip

\begin{cor}\label{corollary:example}
For the permutation $\tau_0$ from Example \ref{example:rearrangement_type},
$p\in (1,2)$, and $X:=\ell_p$ one has 
\[  \noo \Id_X\otimes T_{\tau_0}: L^p_{X,0} \to L^p_X \rrm < \infty \]
but
\[  \noo \Id_X\otimes T_{\tau_0}: L^q_{X,0} \to L^q_X \rrm = \infty
    \sptext{1}{for all}{1}
    q\in (p,2]. \]
\end{cor}

\begin{proof}
The first relation follows from Fubini's theorem and the Semenov 
condition. On the other side, $X=\ell_p$ is not of type $q$ as long 
as  $q\in (p,2]$ so that $T_{\tau_0}$ fails to be bounded in $L^q_{X,0}$.
\end{proof}  


\section{Maurey's extrapolation method and the Semenov condition}
\label{sec:maurey}

By Corollary \ref{corollary:example} we have seen that an extrapolation
from $p$ to $q$ fails in general if $q\in (p,2]$. Here one should note
that the boundedness of $\Id_X\otimes T_\tau:L^p_{X,0}\to L_X^p$ implies the boundedness
of $T_\tau:L^p_0\to L^p$, hence the Semenov condition.
The aim of this section is to show that, by Maurey's extrapolation 
method \cite{Mau}, one has an extrapolation from
$p$ to $q$ in the case that $q\in (1,p)$.
\medskip

\begin{definition}\rm
Let $\tau:\cD\to\cD$ be a permutation with $|\tau(I)|=|I|$.
An operator $A$ which maps $f \in L_{X,0}^1(\ftnd_n)$ into a non-negative
random variable $A(f): [0,1)\to [0,\infty)$ and which is homogeneous
(i.e. $A(\mu f) = |\mu| A(f)$, $\lambda$-a.s., for all $\mu\in\bK$), 
where $n\ge 1$, is $\tau$-monotone with
constant $c>0$ provided that one has, $\lambda$-a.s., that
\begin{equation}\label{eqn:monotonicity_condition}
    A\left ( \sum_{k=1}^n \gamma_k d_k  \right )
\le c \sup_{k=1,...,n} | P_{k-1,\tau} (\gamma_k)|
    A\left ( \sum_{k=1}^n d_k \right )
\end{equation}
for all
\[   d_k(t) 
  = \sum_{I\in \cD_{k-1}} a_I h_{I}(t), \quad a_I\in X, \]
and non-decreasing $(\gamma_k)_{k=1}^n$ with
\[   \gamma_k (t)
   = \sum_{I\in\cD_{k-1}} \gamma_k(I) I_{I}(t), \quad \gamma_k(I) 
     \ge 0, \]
where
$     P_{k-1,\tau} (\gamma_k) 
   := \sum_{I\in\cD_{k-1}} \gamma_k(I) I_{\tau(I)}(t)$.
\end{definition}
\bigskip

Note that $P_{k,\tau}(\gamma)$ is correctly defined for all 
$\gamma:[0,1)\to\bR$
that are constant on the dyadic intervals of length $2^{-k}$.

\begin{theorem}\label{theorem:one_sided_extrapolation_new_new}
For a permutation $\tau:\cD\to\cD$ with $|\tau(I)|=|I|$ 
the following assertions are equivalent:
\begin{enumerate}[{\rm (i)}]
\item The permutation $\tau$ satisfies the Semenov condition
      {\rm (\ref{eqn:Semenov_condition})}.
\item For all $1<q<p<\infty$, Banach spaces $X$, $n=1,2,....$, and 
      $\tau$-monotone operators $A$, defined on $L^1_{X,0}(\ftnd_n)$, 
      with constant $c>0$ one has that 
      \[       \| A : L^q_{X,0}(\ftnd_n)\to L^q([0,1)) \|  
         \le d \| A : L^p_{X,0}(\ftnd_n)  \to L^p([0,1)) \| \]
      where $d=d(p,q,c)>0$ and
      \[    \| A \|_r
          = \| A : L^r_{X,0}(\ftnd_n)\to L^r([0,1)) \|
         := \sup \left \{ \|A(f)\|_{L^r} : \| f \|_{L_{X,0}^r} \le 1 \right \}. \]

\end{enumerate}
\end{theorem}
\bigskip

Before we give the proof of 
Theorem~\ref{theorem:one_sided_extrapolation_new_new} 
we apply it to our original extrapolation problem. 
\smallskip

\begin{cor}\label{cor:one_sided_extrapolation_under_UMD}
Let $X$ be a $\umd$-space and let $ \tau: \cD \to \cD $ be a permutation
such that 
\[ |\tau ( I) | = |I| . \]
If, for some $p\in (1,2)$, one has that
\[ \Id_X \otimes T_\tau  : L^p_{X,0} \to  L^p_X \]
is bounded, then 
\[ \Id_X \otimes T_\tau  : L^q_{X,0} \to  L^q_X \]
is bounded for all  $q\in (1,p)$.
\end{cor}
\smallskip

\begin{proof}
Because our assumption implies that 
$\Id_\bK\otimes T_\tau:L^p_0\to L^p$ is bounded it has to satisfy the
Semenov condition. We fix $n\ge 1$ and apply the previous theorem to
the operator $A$ defined,
for  $d_k = \sum_{I \in \cD_{k-1}} a_I h_I$ with $ a_I \in X$, as
      \[    A  \left (\sum_{k=1}^n d_k \right ) 
         := \int_\Omega \left \|  (\Id_X \otimes T_\tau)
            \left(\sum_{k=1}^n r_k(\omega) d_k \right)\right \|_X
            d \prob(\omega) \]
where $r_1,...,r_n$ are independent Bernoulli random variables.
It is easy to see that $A$ satisfies 
(\ref{eqn:monotonicity_condition}) with $c=1$. Moreover by the 
$\umd$-property we have 
\[        \left \| A \left ( \sum_{k=1}^n d_k \right ) \right \|_{L^p}
    \sim  \left 
\| (\Id_X \otimes T_\tau) 
   \left(\sum_{k=1}^n d_k \right) \right \|_{L_X^p}, \]
where the multiplicative constants do not depend on $n$.
Hence Theorem~\ref{theorem:one_sided_extrapolation_new_new} yields 
the assertion.
\end{proof}
\smallskip

The maximal inequality of  the following 
Proposition \ref{proposition:maximal_inequality_condition_Semenov} 
provides the link between  rearrangements satisfying 
Semenov's  condition  and Maurey's extrapolation technique
in \cite{Mau}.
\newpage

\begin{prop}\label{proposition:maximal_inequality_condition_Semenov}
Assume that Semenov's condition $(\ref{eqn:Semenov_condition})$ is satisfied
for a permutation $\tau$ with $|\tau(I)|=|I|$ and that 
$0\le Z_0 \le Z_1 \le \cdots \le Z_n$ is a sequence of functions
$Z_k:[0,1)\to [0,\infty)$, where $Z_k$ is constant on all dyadic
intervals of length $1/2^k$. Then one has that 
\[     \int_0^1 \sup_{k=0,...,n} (P_{k,\tau} (Z_k))(t) dt
   \le \kappa \int_0^1 Z_n(t) dt. \]
\end{prop}
\medskip
\begin{proof}
Let $\Delta_0:= Z_0$ and 
$\Delta_k := Z_k - Z_{k-1}$ for $k=1,...,n$, and let us write
\[ 
\Delta_k = \sum_{I\in \cD_k} a_I 1_I
\]
with $a_I \ge 0$. Fix $k\in \{0,...,n\}$ and observe that, point wise, 
\[ P_{k,\tau} 1_I  \le 1_{\tau (Q(I))^*}
   \sptext{1}{with}{1}
   Q(I) = \{ K\subseteq I : K \in \cD \} \]
for $I\in \cD_{k'}$ with $k'=0,...,k$ (note that $1_I$ is constant
on the dyadic intervals of length $2^{-k}$ so that we may apply
$P_{k,\tau}$). This implies that
\[     P_{k,\tau}\left (\sum _{k'=0}^k \sum_{I\in \cD_{k'}} a _I 1_I 
       \right )
   \le \sum _{k'=0}^k \sum_{I\in \cD_{k'}} a _I  1_{\tau (Q(I))^*}. \]
Because the expression on the right-hand side is monotone in $k$
we conclude that
\[     \sup_{k=0,...,n} 
       P_{k,\tau} \left ( \sum _{k'=0}^k \sum_{I\in \cD_{k'}} a _I  1_I 
       \right ) 
   \le \sum_{k'=0}^n    \sum_{I\in\cD_{k'}} a _I 1_{\tau (Q(I))^*}. \]
Integration gives 
\[     \int_0^1 \left [ 
       \sup_{k=0,...,n} 
       P_{k,\tau}\left ( \sum _{k'=0}^k \sum_{I\in\cD_{k'}} 
       a _I 1_I \right )(t) \right ] dt
   \le \sum_{k'=0}^n \sum_{I\in \cD_{k'}} a _I |\tau (Q(I))^*|. \]
Our hypothesis gives $|\tau (Q(I))^*| \le \kappa |I|$ so that
\[ 
    \sum_{k'=0}^n \sum_{I\in \cD_{k'}} a _I |\tau (Q(I))^*|
\le \kappa \sum_{k'=0}^n \sum_{I\in \cD_{k'}} a _I |I| 
 =  \kappa \int_0^1  \left [ \sum_{k=0}^n \Delta _k(t) \right ] dt
\]
and we are done because
\[   \int_0^1 \left [ 
         \sup_{k=0,...,n} 
         P_{k,\tau}\left ( \sum _{k'=0}^k \sum_{I\in \cD_{k'}} 
                           a _I 1_I \right ) (t) 
               \right ] dt
   = \int_0^1 \sup_{k=0,...,n} (P_{k,\tau} Z_k)(t) dt \]
and
\[    \int_0^1  \left [ \sum_{k=0}^n \Delta _k(t) \right ] dt
   =  \int_0^1 Z_n(t) dt. \]
\end{proof}
\bigskip

\begin{proofof}{Theorem \ref{theorem:one_sided_extrapolation_new_new}.}
${\rm (i)}\Longrightarrow {\rm (ii)}$
We let $\frac{1}{q}=\frac{1}{r} + \frac{1}{p}$ and
\[ d_k := \sum_{I\in\cD_{k-1}} \alpha_I h_I
   \sptext{1}{so that}{1}
   T_\tau d_k = \sum_{I\in \cD_{k-1}} \alpha_I h_{\tau(I)}. \]
Define 
$X_0:=0$,
$X_k:=d_1+\cdots+d_k$ for $k=1,...,n$, 
$X^*_k := \sup_{l=0,...,k} \| X_l\|_X$ for $k=0,...,n$,
${^*}X_k:=X^*_{k-1} + \sup_{l=1,...,k} \|d_l\|_X$ 
for $k=1,...,n$, 
\[ \gamma_k    := ({^*}X_k + \delta)^\alpha \]
for some $\delta>0$,
\[ \alpha      := 1 - \frac{q}{p}, \]
and 
\[ \beta_k  := P_{k-1,\tau} \gamma_k. \]
By definition we have that 
\[   \frac{T_\tau(d_k)}{\beta_k} 
   = T_\tau \left ( \frac{d_k}{\gamma_k} \right ). \]
>From the monotonicity assumption on the operator $A$ it follows that
\[      \noo A\left ( \sum_{k=1}^n d_k \right ) \rrm_{L^q} 
    \le c \| \beta^*_n\|_{L^r}
        \left \| A \left  (\sum_{k=1}^n \frac{d_k}{\gamma_k} \right ) \right \|_{L^p}
    \le c \| A \|_p  \| \beta^*_n\|_{L^r} 
        \left \| \sum_{k=1}^n \frac{d_k}{\gamma_k} \right \|_{L^p_X}. \]
>From \cite[Lemma A]{Mau} we know that
\equa
      \left \| \sum_{k=1}^n \frac{d_k}{\gamma_k} \right \|_{L^p_X}
&\le& \frac{p}{q} \left ( \ev (^*X_n+\delta)^q \right )^\frac{1}{p} \\
&\le& \frac{p}{q} 3^\frac{q}{p}  \left ( \ev (X_n^*+\delta)^q \right )^\frac{1}{p}.
\tion
Finally, applying 
Proposition \ref{proposition:maximal_inequality_condition_Semenov}
we get
\begin{multline*}
      \| \beta^*_n \|_{L^r}^r
    = \int_0^1 \sup_{k=1,...,n} | (P_{k-1,\tau}(\gamma_k))(t)|^r dt 
    = \int_0^1 \sup_{k=1,...,n}   (P_{k-1,\tau}(|\gamma_k|^r))(t) dt \\
   \le \kappa \int_0^1 | \gamma_n(t) |^r dt 
    =  \kappa \int_0^1 | ^*X_n(t) +\delta |^{\alpha r} dt 
   \le 3^{\alpha r} \kappa \int_0^1 | X_n^*(t) +\delta |^{\alpha r} dt.
\end{multline*}
Combining all estimates, we get
\[     \noo A\left ( \sum_{k=1}^n d_k \right ) \rrm_{L^q} 
   \le c \| A \|_p 3^\alpha \kappa^\frac{1}{r} 
       \left ( \ev | X_n^* + \delta|^{\alpha r} \right )^\frac{1}{r}
       \frac{p}{q} 3^\frac{q}{p} 
       \left ( \ev | X_n^* + \delta|^q \right )^\frac{1}{p}. \]
By $\delta\downarrow 0$ and Doob's maximal inequality this implies 
\[     \noo A\left ( \sum_{k=1}^n d_k \right ) \rrm_{L^q} 
   \le c \| A \|_p \frac{3p}{q-1} \kappa^\frac{1}{r} 
       \| d_1 + \cdots + d_n \|_{L_X^q}. \]
${\rm (ii)}\Longrightarrow {\rm (i)}$
We fix $X=\bK$, $n\in \{1,2,... \}$, and a permutation 
$\tau$ with $|\tau(I)| = |I|$. Let 
$   A \left ( \sum_{k=1}^nd_k\right ) 
 := \left ( \sum_{k=1}^n (T_\tau d_k)^2\right )^\frac{1}{2}$
which is $\tau$-monotone with constant $c=1$.
Clearly, $\| Af \|_{L^2} = \| f\|_{L^2}$. If we have an extrapolation to some
$q\in (1,2)$, then by  the square function inequality the 
usual permutation operator is bounded in  $L^q$ with a constant not 
depending on $n$, so that by Semenov's theorem \cite{Sem} 
condition  {\rm (\ref{eqn:Semenov_condition})} has to be satisfied.
\end{proofof}
 

\section{Extrapolation and the  Carleson condition}
\label{sec:carleson}

In this  section we consider rearrangement operators induced by 
bijections $\tau : \cD \to \cD $ that preserves the Carleson packing 
condition, that is there is an $A\ge 1$ such that
\[      \frac{1}{A} \y \cE \yy 
   \le  \y \tau(\cE) \yy 
   \le A \y \cE \yy  \]
for all non-empty $\cE\subseteq \cD$. In particular, we do not rely anymore 
on the a-priori hypothesis that $ | \tau (I)|  =| I |$.
The corresponding extrapolation results are formulated in
Corollary \ref{corollary:downwards_extrapolation},
Corollary \ref{corollary:downwards_extrapolation_semenov}, and
Theorem \ref{theorem:extraisolp}, where we obtain in
Corollary \ref{corollary:downwards_extrapolation_semenov} an alternative 
proof of Corollary \ref{cor:one_sided_extrapolation_under_UMD}
that works without  $X $ being a $\umd$-space.
To shorten the notation we let 
$\cD_0^N := \bigcup_{k=0}^N \cD_k$ for $N\ge 0$.
Because we use complex interpolation we shall assume that
all Banach spaces are complex.
\bigskip

We start with a technical condition which ensures a one-sided
extrapolation. The condition will be justified by
Examples \ref{example:Semenov} and \ref{example:propertyP} 
below.

 \begin{definition}\label{definition:cqkappa}\rm
Let $X$ be a Banach space,
$\tau: \cD_0^N \to \cD_0^L$ be an injection, 
$\gamma_I>0$ for $I\in \cD_0^N$, $p\in (1,\infty)$, and $\kappa>0$.
We say that condition $C(X,p,\kappa)$ is satisfied,
provided that for all $J_0\in \cD_0^N$ there is a decomposition
\[   \left \{  I\in \cD_0^N : I \subseteq J_0 \right \}
   = \bigcup_i \cK_i, \]
$\cK_i\not = \emptyset$,
such that the following is satisfied:
\begin{enumerate}[(C1)]
\item $\sum_i |\cK_i^*| \le \kappa |J_0|$.
\item For $1=\frac{1}{p}+\frac{1}{q}$ and 
      \[ \beta_i := \sup \left \{ 
         \left \| \sum_{I\in\cK_i} \gamma_I^\frac{1}{q} a_I h_I
         \right \|_{L^p_X}^q:
         \left \| \sum_{I\in\cK_i} a_I h_I \right \|_{L^p_X} = 1 \right \}\]
      one has that
      $ \sum_i \beta_i |\tau(\cK_i)^*|  
                 \le \kappa |J_0|$. 
\item There exists $p_* \in [p,\infty)$ such that
      \[     \left ( \sum_i \left \| \sum_{I\in\cK_i} a_I h_I
                     \right \|_{L^{p_*}_X}^{p_*}\right )^\frac{1}{p_*}
         \le \kappa \left \| \sum_{J_0 \supseteq I \in\cD_0^N} 
             a_I h_I \right \|_{L^{p_*}_X}. \]
\end{enumerate}
\end{definition}

\begin{example}\label{example:Semenov} \rm
We assume that $\tau:\cD\to\cD$ with $|\tau(I)|=|I|$ satisfies 
the Semenov condition (\ref{eqn:Semenov_condition}) with constant 
$\kappa\in [1,\infty)$, restrict $\tau$ to $\tau_N:\cD_0^N\to\cD_0^N$, 
and take $\gamma_I=1$ for all $I\in \cD_0^N$. Let $X$ be arbitrary,
$p\in (1,\infty)$, and $J_0 \in \cD_0^N$.
Because of
\[     \left |\bigcup_{J_0 \supseteq I\in\cD_0^N} \tau_N(I)\right | 
   \le \kappa |J_0| \]
we can take
\[  \cK_1 := \left \{  I\in \cD_0^N : I \subseteq J_0 \right \} \]
and conditions (C1), (C2), and (C3) (for any $p^*$) are satisfied
with constant $\kappa$ uniformly in $N$.
\end{example}

\begin{example}\rm\label{example:propertyP}
Let $ \tau : \cD \to \cD $ be a bijection and assume that there is an 
$A\ge 1$ such that
\[     \frac{1}{A}  \y      \cE   \yy
   \le              \y \tau(\cE ) \yy
   \le           A  \y      \cE   \yy \]
for all non-empty $\cE \sbe \cD$.
Let $X$  be a \umd-space and $\gamma_I := |I|/|\tau(I)|$.
As shown in \cite[Theorem 1]{pfxm1}, the permutation $\sigma=\tau^{-1}$
satisfies the following property P:
There exists an $M>0$ such that for all dyadic intervals $J_0 \in\cD$ there exists a 
decomposition as disjoint union 
\[   \{ I \in \cD : I \sbe J_0 \}  
   = \sigma(\cD)\cap J_0
   = \bigcup_i \s (\cL_i) \cup \bigcup_i \cE_i \]
such  that 
\begin{enumerate}[(1)]
\item $\left \y \bigcup_i \cE_i \right \yy \le M$,
\item $\sup_{K \in \cL_i} \frac{|\s(K)|}{|K|} \le M \frac{ |\s (\cL _i)^*|  
       + |\cE_i^ * |}{|\cL_i^*|}$ for $\cL_i\not = \emptyset$,
\item $\sum _i |\s (\cL _i)^* | \le M |J_0|$.
\end{enumerate}
Now we check the counterparts of (C1), (C2), and (C3) for the 'infinite'
permutation $\tau$.

Condition (C3): As  $X$ is a \umd-space (and therefore super-reflexive)
there is a $p_0\in [2,\infty)$ such that 
for all $p_*\in [p_0,\infty)$ the space $X$ has cotype $p_*$. This cotype and 
the \umd-property imply (C3) (the constant may depend on $p_*$). 

Condition (C1): We write 
\[    \bigcup_i \cE_i
    = \left \{\widetilde{I}_1,\widetilde{I}_2,...\right \} 
   \sptext{1}{and}{1}
      \widetilde{\cL}_j
   := \{ \tau(\widetilde{I}_j)\} \]
so that
\[   \left \{  I\in \cD : I \subseteq J_0 \right \}
   = \bigcup_i \sigma (\cL_i) \cup
     \bigcup_j \sigma (\widetilde{\cL}_j)
   =:  \bigcup_i \cK_i \cup
       \bigcup_j \widetilde{\cK}_j. \]
Now
\[
     \sum_i |\cK_i^*|         + \sum_j |\widetilde{\cK}_j^*| 
   = \sum_i |\sigma(\cL_i)^*| + \sum_j |\widetilde{I}_j| 
  \le M | J_0| +  \y \bigcup_i \cE_i \yy |J_0|
  \le 2 M | J_0|. \]
Condition (C2): let $p\in (1,\infty)$ be arbitrary and recall that 
\[ \beta_i =   \sup \left \{ 
      \left \| \sum_{I\in\cK_i} \gamma_I^\frac{1}{q} a_I h_I
      \right \|_{L_p^X}^q:
      \left \| \sum_{I\in\cK_i} a_I h_I \right \|_{L_p^X} = 1 \right \}, \]
where we assume that the sums over $I$ are finitely supported,
and let
\[ \widetilde{\beta}_j :=   \sup \left \{ 
      \left \| \sum_{I\in\widetilde{\cK}_j} \gamma_I^\frac{1}{q} a_I h_I
      \right \|_{L_p^X}^q:
      \left \| \sum_{I\in\widetilde{\cK}_j} a_I h_I \right \|_{L_p^X} = 1 \right \} 
   = \gamma_{\widetilde{I}_j}. \]
Because $\gamma_I = |I|/|\tau(I)|$, the $\umd$-property of $X$ gives
\[ \beta_i \le \umd_p(X)^q\sup_{I\in\cK_i} \frac{|I|}{|\tau(I)|}. \]
Since 
\[     \sup_{I\in\cK_i} \frac{|I|}{|\tau(I)|}
   \le M \frac{|\cK_i^*|+|\cE_i^*|}{|\tau(\cK_i)^*|} \]
for $\cL_i\not = \emptyset$ we get 
\equa
      \sum_i \beta_i |\tau(\cK_i)^*| 
&\le& \umd_p(X)^q \sum_i \sup_{I\in\cK_i} \frac{|I|}{|\tau(I)|} |\tau(\cK_i)^*| \\
&\le& \umd_p(X)^q \sum_i M \frac{|\cK_i^*|+|\cE_i^*|}{|\tau(\cK_i)^*|} 
      |\tau(\cK_i)^*| \\
& = & M \umd_p(X)^q \sum_i [|\cK_i^*|+|\cE_i^*|] \\
&\le& 2 M^2 \umd_p(X)^q |J_0|.
\tion 
In the same way,
\[    \sum_j \widetilde{\beta}_j |\tau(\widetilde{\cK}_j)^*| 
   =  \sum_j |\widetilde{I}_j| 
  \le M |J_0|. \]
Finally, if we restrict $\tau$ to 
$\tau_N:\cD_0^N\to \cD_0^{L_N}$ with $L_N$ chosen such that
$\tau(\cD_0^N) \subseteq \cD_0^{L_N}$, then (C1), (C2), and (C3) are
satisfied with the same constant uniformly in $N$. 
\end{example}

In the following we use the notation 
\[ L_X^r(\cD_0^N) := L_{X,0}^r(\ftnd_{N+1}), \hspace{.5em}
   H_X^{1,at}(\cD_0^N) := H_X^{1,at}(\ftnd_{N+1}), \]
and $H_X^1(\cD_0^N) := H_X^1(\ftnd_{N+1})$ 
for $N=0,1,...$ to avoid a permanent shift in $N$
because we are working with the sets $\cD_0^N$ rather 
than with the $\sigma$-algebras $\ftnd_N$. 
Now fix Banach spaces $X$ and $Y$ and a bounded linear operator 
$S:X\to Y$,  and define  the  family of operators 
$A_p : L^p_X \left (\cD_0^N \right )
 \to   L^p_Y \left (\cD_0^L \right )$
by
\[    A_p \left ( \sum_{I\in \cD_0^N} a_I h_I \right )  
   := \sum_{I\in\cD_0^N} S a_I \gamma_I^\frac{1}{p} h_{\tau(I)}, \]
where $\gamma_I>0$.
We aim at    extrapolation theorems  for 
this  family of operators and extrapolate - under the condition
$C(X,p,\kappa)$ - from $L^p$ downwards to $H^1$ in a first step:

\begin{theorem}\label{theorem:generalextrapolation}
If  $p\in (1,\infty)$  and if assumption $C(X,p, \kappa)$ holds, then
\[      \| A_1: H^1_X(\cD_0^N)\to H^1_Y(\cD_0^L) \|
   \le  \frac{18 p}{p-1} \kappa^{1+\frac{1}{q_*}} 
        \| A_p: L^p_X(\cD_0^N) \to L^p_Y(\cD_0^L) \| \]
where $1=(1/p_*)+(1/q_*)$ and $p_*$ is taken from the definition
of  $C(X,p, \kappa)$.
\end{theorem}

\begin{proof}
Let $1=\frac{1}{p}+\frac{1}{q}$ and
let $a\in H^{1,at}_X(\cD_0^N)$ be an atom with associated stopping time
$\nu$ (like in Definition \ref{definition:hardy_spaces})
and assume first that $\{ \nu<\infty \} = J_0\in \cD_0^N$.
For $J_0$ we choose the sets $\cK_i$ like in Definition \ref{definition:cqkappa}.
Moreover, we use
\[ D_q a := \sum_{I\in\cD_0^N} \gamma_I^\frac{1}{q} a_I h_I 
   \sptext{1}{and}{1}
   a_i := \sum_{I\in\cK_i} a_I h_I \]
for $a=\sum_{I\in\cD_0^N} a_I h_I$ and
\[ \beta_i :=  \sup \left \{ 
               \left \| \sum_{I\in\cK_i} \gamma_I^\frac{1}{q} a_I h_I
               \right \|_{L^p_X}^q :  
               \left \| \sum_{I\in\cK_i} a_I h_I \right \|_{L^p_X} = 1 \right \} . \]
We get that
\pagebreak
\equa
      \| A_1 a \|_{H^1_Y} 
&\le& \sum_i \| A_1 a_i \|_{H^1_Y} \\
& = & \sum_i \| A_p D_q a_i \|_{H^1_Y} \\
&\le& \sum_i |\tau(\cK_i)^*|^\frac{1}{q}
      \left \| A_p D_q a_i \right \|_{H^p_Y} \\
&\le& \frac{p}{p-1} \sum_i |\tau(\cK_i)^*|^\frac{1}{q}
      \left \| A_p D_q a_i \right \|_{L^p_Y} \\
&\le& \frac{p}{p-1}  \| A_p \| \sum_i |\tau(\cK_i)^*|^\frac{1}{q}
      \left \| D_q a_i \right \|_{L^p_X} \\
&\le&  \frac{p}{p-1} 
      \| A_p \| \sum_i 
      \left [ |\tau(\cK_i)^*| \beta_i \right ]^\frac{1}{q}
      \left \| a_i \right \|_{L^p_X} \\
&\le& \frac{p}{p-1} \| A_p \| \sum_i 
      \left [ |\tau(\cK_i)^*|    \beta_i 
      \right ]^\frac{1}{q}
      |\cK_i^*|^{\frac{1}{p}-\frac{1}{p_*}}
      \left \| a_i \right \|_{L^{p_*}_X} \\
&\le& \frac{p}{p-1} \| A_p \| 
      \left ( \sum_i \left | \left [ |\tau(\cK_i)^*|
              \beta_i \right ]^\frac{1}{q} 
              |\cK_i^*|^{\frac{1}{p}-\frac{1}{p_*}} \right |^{q_*} 
      \right )^\frac{1}{q_*} \\
&   & \hspace*{14em}  
     \left ( \sum_i \left \| a_i \right \|_{L^{p_*}_X}^{p_*} 
     \right )^{\frac{1}{p_*}}
\tion
with $1=\frac{1}{q_*}+\frac{1}{p_*}$. Letting $r:= \frac{q}{q_*}$ and 
$1=\frac{1}{r} + \frac{1}{s}$ we obtain that
\[     \sum_i \left | \left [ |\tau(\cK_i)^*|
              \beta_i \right ]^\frac{1}{q} 
              |\cK_i^*|^{\frac{1}{p}-\frac{1}{p_*}} \right |^{q_*} 
   \le \left ( \sum_i \left [ |\tau(\cK_i)^*|
              \beta_i \right ] \right )^\frac{1}{r}
       \left ( \sum_i |\cK_i^*| \right )^\frac{1}{s} 
   \le \kappa |J_0| \]
(with the obvious modification for $q=q_*$) and
\equa
      \| A_1 a \|_{H^1_Y} 
&\le& \frac{p}{p-1} \kappa^{\frac{1}{q_*}} \| A_p \| 
      |J_0|^\frac{1}{q_*}  \left ( \sum_i \left \| a_i \right \|_{L^{p_*}_X}^{p_*} 
      \right )^{\frac{1}{p_*}} \\
&\le& \frac{p}{p-1} \kappa^{1+\frac{1}{q_*}} \| A_p \| 
      |J_0|^\frac{1}{q_*}  \| a  \|_{L^{p_*}_X} \\
&\le& \frac{p}{p-1} \kappa^{1+\frac{1}{q_*}} \| A_p \| 
      |J_0| \| a  \|_{L^\infty_X} \\
&\le& \frac{p}{p-1}  \kappa^{1+\frac{1}{q_*}} \| A_p \|.
\tion
It is not difficult to check that any atom $a\in H_X^{1,at}(\cD_0^N)$ can be written as
finite convex combination of atoms considered in this proof so far.
Using this and (\ref{eqn:H1_H1at}) we end up with
\[ \| A_1 a \|_{H^1_Y} \le  \frac{p}{p-1}  \kappa^{1+\frac{1}{q_*}} \| A_p \|
                            \| a \|_{H^{1,at}_X}
                        \le \frac{18 p}{p-1}  \kappa^{1+\frac{1}{q_*}} \| A_p \|
                            \| a \|_{H^{1}_X} \]
for all $a\in H_X^1(\cD_0^N)$.
\end{proof}

Now we interpolate between  $H^1$ and  $L^p$:
\begin{lemma}\label{lemma:complexinterpolation}
Let $1<q<p<\infty$ and $\frac{1}{q} = \frac{1-\theta}{1} + \frac{\theta}{p} $.
If $Y$ is a \umd-space, then one has
\begin{multline*}
          \| A_q : L^q_X(\cD_0^N) \to L^q_Y(\cD_0^L) \| \\
    \le c \| A_1 : H^1_X(\cD_0^N) \to H^1_Y(\cD_0^L) \|^{1-\theta} 
          \| A_p : L^p_X(\cD_0^N) \to L^p_Y(\cD_0^L) \|^{\theta}
\end{multline*}
where $c>0$ depends at most on $Y$, $p$, and $q$.
In the case $\gamma_I\equiv 1$ the \umd-property of $Y$ is not needed and $c>0$
does not depend on $Y$.
\end{lemma}

\begin{proof}
Because we work with probability spaces consisting of a finite number of 
atoms only, we can replace (for simplicity)  $X$ and $Y$ by finite 
dimensional subspaces $E\subseteq X$ and $F\subseteq Y$ such that
$S(E)\subseteq F$,
where we will see that the constant $c$ can be chosen uniformly for all
subspaces $E$ and $F$.
The family $(A_q)_{q\in [1,p]}$ is embedded into an analytic family of
operators.  Let $V$ denote the vertical strip 
$V = \{ x+it \, : \, x \in (0,1), t \in \bR \}$ and let
\[ J_z(a) := \sum_{I\in\cD_0^N} S a_I \gamma_I^{1-z(1- \frac{1}{p})} h_{\tau(I)}. \]
As $ \frac{1}{q} = \frac{1-\theta}{1} + \frac{\theta}{p} $ we have
\[ J_\theta = A_q. \]
Since
\[    \Re \left (1 - i t \left ( 1 - \frac{1}{p}\right ) \right ) 
    = 1
    \qquad \textrm{and} \qquad 
      \Re \left (1 - (1+i t)\left (1 - \frac{1}{p}\right ) \right ) 
    = \frac{1}{p}, \]
we have
\begin{equation}\label{eqn:umpp}
         \|J_{1+it}(f) \|_{L^p_F(\cD_0^L)} 
    \le  2 \umd_p(Y) \| A_p(f) \|_{L^p_F(\cD_0^L)}
\end{equation}
and
\begin{equation}\label{eqn:umdh1}
        \| J_{it}(f) \|_{H^1_F(\cD_0^L)}         
    \le c \| A_1(f) \|_{H^1_F(\cD_0^L)}
\end{equation}
for some $c>0$ depending on $Y$ only.
The latter estimate ($Y$ is a \umd-space) is folklore and can be 
derived in various ways. For example, one can follow 
\cite[Remarque 2]{Mau}.
Following the proof that the complex interpolation method with parameter $\theta$
yields an exact interpolation functor of exponent $\theta$, for example presented 
in \cite[Theorem 4.1.2]{B-L}, we get that 
\equa
&   & \| J_{\theta}(f) \|_{(H_F^1(\cD_0^L),L^p_F(\cD_0^L))_\theta} \\
&\le& \sup_{t\in\bR} 
       \| J_{it} : H^1_E(\cD_0^N) \to H^1_F(\cD_0^L) \|^{1-\theta} 
       \sup_{t\in\bR} 
       \| J_{1+it} : L^p_E(\cD_0^N) \to L^p_F(\cD_0^L) \|^{\theta} \\
&   & \hspace*{21em}
      \| f \|_{(H_E^1(\cD_0^N),L^p_E(\cD_0^N))_\theta} \\
&\le& c^{1-\theta}  (2 \umd_p(Y))^\theta
      \| A_1 : H^1_E(\cD_0^N) \to H^1_F(\cD_0^L) \|^{1-\theta} \\
&   & \hspace*{9em}
      \| A_p : L^p_E(\cD_0^N) \to L^p_F(\cD_0^L) \|^{\theta} 
      \| f \|_{(H_E^1(\cD_0^N),L^p_E(\cD_0^N))_\theta}
\tion
where $(Z_0,Z_1)_\theta$ denotes the interpolation space obtained by
the complex method as in \cite[p. 88]{B-L}. Using
\begin{equation}\label{eqn:interpolation_hardy}
   ( H^1_E(\cD_0^N), L^p_E(\cD_0^N)_\theta =   L^q_E(\cD_0^N) 
   \sptext{1}{and}{1}
   ( H^1_F(\cD_0^L), L^p_F(\cD_0^L)_\theta =   L^q_F(\cD_0^L)
\end{equation}
with  multiplicative constants not depending on $(N,L,X,Y)$
we arrive at our assertion.
In the case $\gamma_I=1$ we have 
$J_{it} = A_1$ and $J_{1+it} = A_p$ so that the \umd-property
in (\ref{eqn:umpp}) and (\ref{eqn:umdh1}) is not needed.
The equivalences (\ref{eqn:interpolation_hardy}) are folklore,
see \cite[p. 334]{Blasco-Xu}. One can deduce them via
the real interpolation method by exploiting 
$(H^1_Z(\cD_0^M),L_Z^r(\cD_0^M))_{\eta,s} = L^s_Z(\cD_0^M)$ for
$\eta\in (0,1)$, $r,s\in (1,\infty)$ with $(1/s)=1-\eta + (\eta/r)$,
$Z\in \{ E,F\}$, and $M\ge 0$, where the multiplicative constants
in the norm estimates depend on $(\eta,r,s)$ only
 (see \cite{Weisz1} and the references therein),
and the connection between the real and complex interpolation method
presented in the second statement of \cite[Theorem 4.7.2]{B-L}, where 
we use that the proof for the first inclusion works as well with 
$\theta_0=0$, $p_0=1$, and $(\overline{A})_{\theta_0,p_0}$ replaced 
by $A_0$. 
\end{proof}
\medskip

\newpage

\begin{cor}\label{corollary:downwards_extrapolation}
Let $ \tau : \cD \to \cD $ be a bijection such that there is an 
$A\ge 1$ with
\begin{equation}\label{eqn:downwards_extrapolation_assumption}
       \frac{1}{A} \y      \cE  \yy
   \le             \y \tau(\cE) \yy  
   \le A           \y      \cE  \yy
\end{equation}
for all non-empty $\cE \sbe \cD$. Furthermore, let $X$ be a \umd-space, 
$\gamma_I := |I|/|\tau(I)|$, and
$1<q<p<\infty$. Then the boundedness of 
\[ \Id_X \otimes T_{p,\tau}: L^p_{X,0} \to L^p_X \]
implies the boundedness of 
\[ \Id_X \otimes T_{q,\tau}: L^q_{X,0} \to L^q_X. \]
In case of $|\tau(I)|=|I|$ the \umd-property is not needed.
\end{cor}
\bigskip
\begin{proof}
(a) For all $N\ge 0$ we choose $L_N\ge 0$ such that 
\[ \tau(\cD_0^N) \subseteq \cD_0^{L_N}. \]
Then we can consider the restrictions $\tau_N:\cD_0^N\to \cD_0^{L_N}$
for $N\ge 0$. According to Example \ref{example:propertyP} the property
$C(X,p,\kappa)$ for some $\kappa>0$ is satisfied uniformly in $N$.
Applying Lemma \ref{lemma:complexinterpolation} and
Theorem \ref{theorem:generalextrapolation} gives that
\equa
&   & \| T_{q,\tau_N} : L^q_X(\cD_0^N) \to L^q_X(\cD_0^{L_N}) \| \\
&\le& c_{(\ref{lemma:complexinterpolation})} 
      \| T_{1,\tau_N} : H^1_X(\cD_0^N) \to H^1_X(\cD_0^{L_N}) \|^{1-\theta} \\
&   & \hspace*{10em}
      \| T_{p,\tau_N} : L^p_X(\cD_0^N) \to L^p_X(\cD_0^{L_N}) \|^{\theta} \\
&\le& c_{(\ref{lemma:complexinterpolation})} 
      \left ( \frac{18 p}{p-1} \kappa^{1+\frac{1}{q_*}}\right )^{1-\theta}
      \| T_{p,\tau_N} : L^p_X(\cD_0^N) \to L^p_X(\cD_0^{L_N}) \| \\
& =:&  c  \| T_{p,\tau_N} : L^p_X(\cD_0^N) \to L^p_X(\cD_0^{L_N}) \| \\
&\le&  c  \| T_{p,\tau} : L^p_{X,0} \to L^p_X \|.
\tion
(b) Now we consider a strictly increasing sequence of integers $B_N\ge 1$ such 
that
\[ \tau(\cD_0^{B_N}) \supseteq \cD_0^N. \] 
For $a=\sum_{I\in\cD} a_I h_I$, where $(a_I)_{I\in\cD}\subseteq X$ 
is finitely supported, we get 
\equa
      \|T_{q,\tau} a\|_{L_X^q}
& = & \sup_N  \| E(T_{q,\tau} a|\ftn_N)\|_{L_X^q} \\
& = & \sup_N  \| E(T_{q,\tau_{B_N}} a_{B_N}|\ftn_N)\|_{L_X^q} \\
&\le& \sup_N  \| T_{q,\tau_{B_N}} a_{B_N} \|_{L_X^q} \\
&\le& \sup_N  \| T_{q,\tau_{B_N}}:L_X^q(\cD_0^{B_N})\to L_X^q(\cD_0^{L_{B_N}})\|
              \| a_{B_N} \|_{L_X^q(\cD_0^{B_N})} \\
&\le& c \| T_{p,\tau} : L^p_{X,0} \to L^p_X \| \|a\|_{L_{X,0}^q}
\tion
where $\tau_{B_N}: \cD_0^{B_N}\to \cD_0^{L_{B_N}}$ is the restriction
of $\tau$ considered in (a) and $a_{B_N}$ the restriction of $a$ to
$\cD_0^{B_N}$.
\end{proof}
\bigskip 

Modifying slightly the first step in the proof of Corollary \ref{corollary:downwards_extrapolation}
we can remove the assumption that $X$ is a \umd-space in 
Corollary \ref{cor:one_sided_extrapolation_under_UMD}:
\medskip

\begin{cor}\label{corollary:downwards_extrapolation_semenov}
Let $X$ be a Banach space and let $ \tau: \cD \to \cD $ be a permutation
such that $|\tau ( I) | = |I|$. Then, for $1<q<p<2$, the boundedness of 
\[ \Id_X \otimes T_\tau  : L^p_{X,0} \to  L^p_X \]
implies the boundedness of  
$\Id_X \otimes T_\tau  : L^q_{X,0} \to  L^q_X$. 
\end{cor}
\smallskip
\begin{proof}
Our assumption implies $\gamma_I=1$ and that $\tau$ satisfies Semenov's condition
with some $\kappa\in [1,\infty)$. By Example \ref{example:Semenov} the restrictions
$\tau_N:\cD_0^N\to\cD_0^N$ satisfy condition $c(X,p,\kappa)$ for all $p\in (1,\infty)$.
Now we can follow the proof of Corollary \ref{corollary:downwards_extrapolation}
with $L_N=B_N=N$ and $\gamma_I=1$ so that the \umd-property in 
Lemma \ref{lemma:complexinterpolation} is not needed.
\end{proof} 
\bigskip

We close with an extrapolation theorem for rearrangement operators that are 
isomorphisms on $L^p_{X,0}$. For real valued rearrangements, i.e. when 
$X = \bR$, the following theorem is well known. It can be obtained by 
different methods, the most direct route \cite{pfxm2} going  via Pisier's
re-norming in $L^p.$
\bigskip

\begin{theorem} \label{theorem:extraisolp}
Let $\tau:\cD\to\cD$ be a bijection and 
$\gamma_I:= |I|/|\tau(I)|$. Assume that $X$ is a \umd-space.
If there exists a $p\in (1,\infty)$ with $p\not = 2$ such that
\begin{equation}\label{eqn:assumption_extraisolp}
         \|\Id_X \otimes T_{p,\tau}: L^p_{X,0} \to L^p_X \|
         \cdot \|\Id_X \otimes T_{p,\tau^{-1}}: L^p_{X,0} \to L^p_X \|
     <   \infty,
\end{equation}
then for each $q\in (1,\infty)$ one has that 
\begin{equation}\label{eqn:conclusion_extraisolp}
               \|\Id_X \otimes T_{q,\tau}: L^q_{X,0} \to L^q_X \|
         \cdot \|\Id_X \otimes T_{q,\tau^{-1}}: L^q_{X,0} \to L^q_X \|
     <   \infty.
\end{equation}
\end{theorem}
\bigskip

\begin{proof}
(a) First we observe that our assumption implies that 
(\ref{eqn:assumption_extraisolp}) holds for $X=\bC$ and $X=\bR$.
If $p\in (2,\infty)$, then \cite[Theorems 2 and 3]{pfxm1} imply
condition (\ref{eqn:downwards_extrapolation_assumption}). In case
of $p\in (1,2)$ duality implies (\ref{eqn:assumption_extraisolp})
for $X=\bR$ and $p$ replaced by the conjugate index $p'\in (2,\infty)$.
Hence we have (\ref{eqn:downwards_extrapolation_assumption}) as well.
\smallskip

(b) From Corollary \ref{corollary:downwards_extrapolation} and (a) we immediately
get (\ref{eqn:conclusion_extraisolp}) for $q\in (1,p)$.
\smallskip

(c) Let $q\in (p,\infty)$. 
It is easy to see that for a bijection $\sigma:\cD\to\cD$ and $r\in (1,\infty)$
the boundedness of 
\[ \|\Id_X \otimes T_{r,\sigma}: L^r_{X,0} \to L^r_X \| 
   \sptext{1}{and}{1}
   \|\Id_{X'} \otimes T_{r',\sigma^{-1}}: L^{r'}_{X',0} \to L^{r'}_{X'} \| \]
are equivalent to each other where $1=(1/r)+(1/r')$
(note, that $X$ is in particular reflexive because of the \umd-property).
Using this observation our assumption (\ref{eqn:assumption_extraisolp})
holds for $p'$ and $X'$ and the conclusion
for $q'\in (1,p')$ and $X'$. By duality we come back to $q$ and $X$. 
\end{proof}


\bibliographystyle{plain}


\textbf{Addresses}
\parindent0em

Department of Mathematics and Statistics \\
P.O. Box 35 (MaD) \\
FIN-40014 University of Jyv\"askyl\"a \\
Finland 

\medskip
Department of Analysis\\
J. Kepler University\\
A-4040 Linz\\
Austria


\end{document}